\documentclass[11pt]{article}

\usepackage{graphicx, amsmath,amsthm,amssymb}
\usepackage[breaklinks=true]{hyperref}
\usepackage[labelfont=bf]{subcaption}
\usepackage[labelfont=bf]{caption}
\usepackage{algorithm,algcompatible,mathtools}

\usepackage[multidot]{grffile}
\usepackage[normalem]{ulem} % just for strikeout

\usepackage{fullpage,etoolbox}
\usepackage{color}
\usepackage[numbers,sort,compress]{natbib}

\newtheorem{theorem}{Theorem}
\numberwithin{theorem}{section}

\newtheorem{proposition}[theorem]{Proposition}
 
\newtheorem{remark}{Remark}

\numberwithin{equation}{section}

\newtoggle{draft}
\togglefalse{draft}

\newcommand{\tf}[1]{\boldsymbol{#1}}

\newcommand{\Phixh}{\hat{\tf \Phi}_x}
\newcommand{\Phiuh}{\hat{\tf \Phi}_u}
\newcommand{\Phix}{\tf\Phi_x}
\newcommand{\Phiu}{\tf\Phi_u}
\newcommand{\Phih}{\hat{\tf \Phi}}
\newcommand{\Dh}{\hat{\tf{\Delta}}}
\newcommand{\D}{{\tf{\Delta}}}
\newcommand{\DA}{\D_{\A}}
\newcommand{\DB}{\D_{\B}}

\newcommand{\A}{\tf A}
\newcommand{\B}{\tf B}
\newcommand{\K}{\tf K}
\newcommand{\Q}{\tf Q}
\newcommand{\x}{\tf x}
\newcommand{\xhat}{\tf{\hat x}}
\newcommand{\uu}{\tf u}
\newcommand{\w}{\tf w}
\newcommand{\what}{\hat{\tf w}}
\newcommand{\eps}{\varepsilon}

\newcommand{\R}{\ensuremath{\mathbb{R}}}

\newcommand{\linfnorm}[1]{\left\|#1\right\|_{\ell_\infty}}
\newcommand{\LTV}{\mathcal{L}_{\mathrm{TV}}}
\newcommand{\LTI}{\mathcal{L}_{\mathrm{TI}}}
\newcommand{\linffnorm}[1]{\left\|#1\right\|_{\ell_\infty\to\ell_\infty}}
\newcommand{\Sp}{\tf S_+}
\newcommand{\Sm}{\tf S_-}
\newcommand{\RHinf}{\mathcal{RH}_\infty}

\title{Robust Performance Guarantees for System Level Synthesis}
\author{Nikolai Matni \vspace{0.0625in}
\\
Electrical and Systems Engineering \\
University of Pennsylvania  \and Anish A. Sarma\vspace{0.0625in}
\\
Computation and Neural Systems \\
California Institute of Technology 
}

\date{\today}

\begin{document}
\maketitle
\begin{abstract}
We generalize the System Level Synthesis (SLS) framework to systems defined by bounded causal linear operators, and use this parameterization to make connections between robust SLS and classical results from the robust control literature.  In particular, by leveraging results from $\mathcal{L}_1$ robust control, we show that necessary and sufficient conditions for robust performance with respect to causal bounded linear uncertainty in the system dynamics can be translated into convex constraints on the system responses.  We exploit this connection to show that these conditions naturally allow for the incorporation of delay, sparsity, and locality constraints on the system responses and resulting controller implementation, allowing these methods to be applied to large-scale distributed control problems -- to the best of our knowledge, these are the first such robust performance guarantees for distributed control systems.  %We also show that for suitably small model uncertainty, sub-optimality bounds on robust performance, as measured with respect to performance achieved by an optimal controller for the true model, can be obtained as a function of the size of the model uncertainty, thus providing a natural bridge between tools from robust control and finite-data guarantees for system identification.
\end{abstract}

\section{Introduction}
\label{sec:introduction}
%!TEX root = main.tex
Robust control seeks to explicitly account for the unavoidable mismatch between the predictions made by a mathematical model of a system and the behavior of the system itself.  In the context of linear control systems, robust control techniques \cite{khammash1990stability,dahleh1994control,zhou1996robust} have proven invaluable in applications ranging from process control to aerospace engineering.  A challenge in robust control is the tension between how detailed a description of model uncertainty is available, and the conservativeness of corresponding computationally tractable analysis and synthesis tools.  Indeed, many of the celebrated tools from robust control, such as the structured singular value (see \cite{packard1993complex}), structured small gain theorems (see Ch 7.2 of \cite{dahleh1994control}), and integral quadratic constraints (IQCs) (see \cite{megretski1997system}), seek to optimally navigate this tension.  However, as of yet, few of these results have been extended in a provably non-conservative manner to the large-scale distributed setting.

Although there is a rich and increasingly mature body of work tackling the distributed optimal control of linear systems (see \cite{2006_Rotkowitz_QI_TAC, 2012_Mahajan_Info_survey, wang2019system,zheng2019equivalence}, and references therein), some of which have robust control interpretations with respect to unstructured norm bounded uncertainty (e.g., \cite{langbort2004distributed,matni2014distributed,lessard2014state,rosinger2017structured,ahmadi2018distributed}), to the best of our knowledge no necessary and sufficient conditions for the robust performance of large-scale distributed controllers exist.  Another branch of related work are methods that use dissipativity theory combined with distributed optimization techniques to derive sufficient conditions for the stability of known interconnected systems, see for example \cite{arcak2016networks,meissen2015compositional,anderson2011dynamical} and the references therein -- while applicable to large-scale systems, these methods can often lead to conservative bounds.

In this paper, we address this gap by leveraging the System Level Synthesis (SLS)  framework \cite{anderson2019system,wang2019system}.  SLS reformulates robust and optimal control problems as an optimization over the achievable closed loop behavior, or \emph{system responses}, of a linear-time-invariant (LTI) dynamical system, and in particular shows that it is necessary and sufficient to constrain these system responses to lie in an affine subspace defined by the dynamics.  This parameterization has been successfully exploited in the context of the distributed optimal control of finite-dimensional LTI systems to scale controller synthesis and implementation techniques to systems of arbitrary size under practically realistic assumptions on the underlying system \cite{wang2014localized, wang2016localized,wang2018separable}.  In order to accommodate general linear time varying uncertainty, we extend the SLS parameterization of internally stabilizing controllers to a class of systems described by causal bounded linear operators, and show how this parameterization can be used to make connections to classical robust synthesis techniques \cite{khammash1990stability,dahleh1994control}.  In doing so, we derive necessary and sufficient conditions for robust performance in terms of \emph{convex} constraints on the system response variables.  We then exploit this connection to show that these necessary and sufficient conditions are equally applicable when additional delay, sparsity, and locality constraints are imposed on the system responses and controller implementation.  To the best of our knowledge, these are the first such necessary and sufficient conditions that are applicable to large-scale uncertain systems and distributed controllers.  In particular, our contributions are:
\begin{itemize}
\item A generalization of the SLS parameterization of stabilizing state-feedback controllers for finite-dimensional LTI systems \cite{wang2019system} to systems with dynamics described by bounded and causal linear operators, wherein we show that constraining system responses to lie in an affine subspace defined by the system dynamics is necessary and sufficient for them to be achievable by a causal linear controller;
\item A generalization of the robustness result of \cite{matni2017scalable} that shows that the  generalized SLS parameterization is stable with respect to perturbations away from the aforementioned achievability subspace, as well as an explicit characterization of the effects of these perturbations on the actual closed loop behavior achieved by the correspondingly perturbed controller implementation;
\item The formulation and solution of a robust performance problem in terms of system response variables for a linear-time-invariant dynamical system subject to bounded and causal linear uncertainty that naturally allows for delay, sparsity, and locality structure to be imposed on the system response and corresponding controller implementation;
\end{itemize}

The rest of the paper is structured as follows: in Section 2, we introduce notation and basic definitions of stability and well-posedness.  In Section 3, we review the SLS parameterization for finite-dimensional LTI systems, and generalize it to a richer class of systems with dynamics described by bounded linear operators.  In Section 4, we consider a robust version of the generalized SLS parameterization derived in Section 3, and provide necessary and sufficient conditions for robust performance in terms of \emph{convex} constraints on the system responses.  In Section 4.1, we highlight how these convex constraints can naturally be integrated with structural constraints on the system responses, such as delay, sparsity, and locality constraints, while still preserving the necessity and sufficiency of the identified conditions.  We end with conclusions and discussions of future work in Section 5.

\section{Notation and Preliminaries}
\label{sec:notation}
%!TEX root = main.tex
We slightly adapt the notation used in \cite{khammash1990stability}.  We use Latin letters to denote vectors and matrices, e.g., $Ax = b$, and bold-face Latin letters to denote signals and operators, e.g., $\tf x = (x_t)_{t=0}^\infty$, and $\tf y = \tf G \tf u$.  We let $\ell_\infty$ denote the space of all bounded sequences of real numbers, i.e., $\tf x = (x_t)_{t=0}^\infty \in \ell_\infty$ if and only if $\sup_t |x_t|<\infty$, in which case we define $\linfnorm{\tf x} = \sup_t |x_t|$.  Similarly, we let $\ell^q_\infty$ denote the space of all $q$-tuples of elements of $\ell_\infty$: if $\tf x = (\tf x^1, \dots, \tf x^q) \in \ell^q_\infty$, then $\linfnorm{\tf x} = \max_{i=1,\dots, q}\linfnorm{\tf x^i}$.  We also define the extended space $\ell^{q}_{\infty,e}$ which is equal to the space of all $q$-tuples of sequences of real numbers.  We let $\tf S_+$ denote the right shift operator such tht if $\tf x = (x_t)_{t=0}^\infty$, then $\tf S_+\tf x = (0,x_0, x_1, \dots)$.  Similarly, we let $\tf S_-$ denote the left shift operator such that $\tf S_- \tf x = (x_1,x_2, \dots)$.  Hence $S_-S_+ = I$, but in general $S_+ S_- \neq I$.  

We let $\LTV^{p,q}$ be the space of all bounded linear causal operators mapping $\ell^q_\infty \to \ell^p_\infty$, and broadly refer to all such operators as $\ell_\infty$-stable.  If $\tf R \in \LTV^{p,q}$, then $\linffnorm{\tf R} := \sup_{\linfnorm{\x}\leq 1}\linfnorm{\tf R \tf x}$, which is the induced operator norm.  Note that each $\tf R \in \LTV^{p,q}$ can be completely characterized by its block lower-triangular pulse response matrix.  We denote by $\LTI^{p,q}$ the subspace of $\LTV^{p,q}$ consisting of time-invariant operators, and further denote by $\RHinf^{p,q}$ the subspace of $\LTI^{p,q}$ consisting of all stable finite-dimensional systems.  When the dimensions can be inferred from context, they will be omitted.  We use $z$ to denote the discrete-time z-transform variable: it follows that the restriction of $\Sp$ to $\LTI$ is given by $\frac{1}{z}$, and similarly the restriction of $\Sm$ to $\LTI$ is given by $z$.  Finally, we recall that an operator $\tf G \in \LTI^{p,q}$ is uniquely characterized by its $z$-transform $\tf G(z) = \sum_{i=0}^\infty z^{-i}G_i$ -- when there is no confusion, we will use $\tf G$ to denote both the operator and its $z$-transform.

\begin{figure}
\centering
\includegraphics[width=.4\columnwidth]{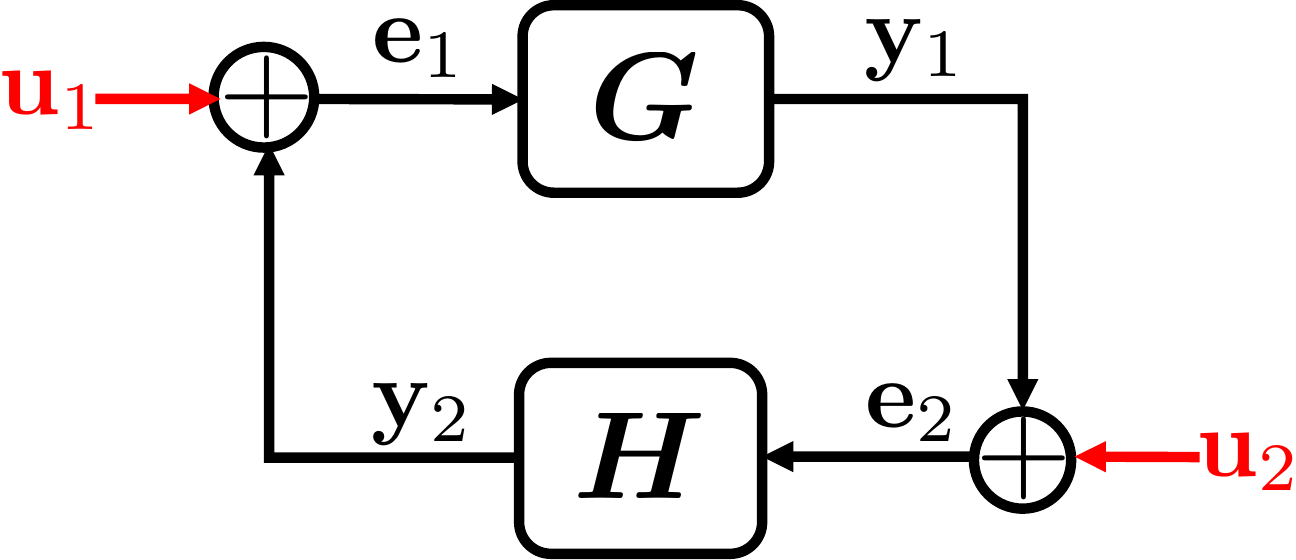}
\caption{A feedback interconnection between systems $\mathbf{G}$ and $\mathbf{H}$.}
\label{fig:well-posed}
\end{figure}

The feedback interconnection as in Fig. \ref{fig:well-posed}, where $\mathcal{G}: \ell_{\infty,e}^p \to \ell_{\infty,e}^q$ and $\mathcal{H}: \ell_{\infty,e}^q \to \ell_{\infty,e}^p$, is \emph{well-posed} if $(I-\mathbf{GH})^{-1}$ exists as a map from $\ell_{\infty,e}^q \to \ell_{\infty,e}^q$, and it is $\ell_\infty$-stable if it is (i) well posed, (ii) the map $(\uu_1,\uu_2) \to (\tf{e}_1,\tf e_2, \tf y_1, \tf y_2)$ takes $ \ell_{\infty}^p \times \ell_{\infty}^q$ into $\ell_{\infty}^p \times \ell_{\infty}^q \times \ell_{\infty}^p \times \ell_{\infty}^q$, and (iii) there exist constants $\alpha_1$ and $\alpha_2$, independent of $\uu_1$ and $\uu_2$, such that $\max\left\{\linfnorm{\tf e_1},\linfnorm{\tf e_2},\linfnorm{\tf y_1},\linfnorm{\tf y_1}\right\} \leq \alpha_1\linfnorm{\uu_1}+\alpha_2\linfnorm{\uu_2}$.  Similarly, a map $\tf G:\ell_{\infty,e}^q \to \ell_{\infty,e}^p$ is $\ell_\infty$-stable if it is causal, takes $\ell_\infty^q$ into $\ell_\infty^p$, and is bounded, i.e., $\linffnorm{\tf G} < \infty$ or equivalently, $\tf G \in \LTV^{p,q}$.  Necessary and sufficient conditions for the interconnection in Fig. 1 to be $\ell_\infty$-stable are then that $\tf G$ and $\tf H$ are $\ell_\infty$-stable, that the interconnection is well-posed, and that $(I-\tf G \tf H)^{-1}$ is $\ell_\infty$-stable.  We note that it is sufficient to verify that only $(I-\tf G \tf H)^{-1}$ is $\ell_\infty$-stable as it is easily shown that this is true if and only if $(I-\tf H \tf G)^{-1}$ is $\ell_\infty$-stable (see Proposition 1, \cite{khammash1990stability}).  Finally we note that in the case of finite dimensional linear-time-invariant (LTI) systems, the interconnection in Fig. \ref{fig:well-posed} is $\ell_\infty$-stable if and only if it is \emph{internally stable}.

\section{Operator System Level Parameterization}
\label{sec:operator}
%!TEX root = main.tex
Let $\tf A \in \LTV^{n,n}$ and $\tf B \in \LTV^{n,p}$, and consider the dynamical system mapping $\ell^n_{\infty,e} \times \ell^p_{\infty,e} \to \ell^n_{\infty,e}$ defined by
\begin{equation}
\tf x = \Sp\tf A\tf x + \Sp \tf B \tf u + \Sp \tf w, \ x_0 = 0.
\label{eq:dynamics}
\end{equation}
As $\Sp\tf A$ is strictly causal, the feedback interconnection defined by the dynamics \eqref{eq:dynamics} is well posed in the sense that $(I-\Sp\tf A)^{-1}$ exists as an operator from $\ell^n_{\infty,e} \times \ell^p_{\infty,e} \to \ell^n_{\infty,e}$ \cite{dahleh1994control}.  We emphasize that although we impose that $\A$ be $\ell_\infty$-stable, this does not imply that the dynamics \eqref{eq:dynamics} are themselves open-loop stable.  Rather, open-loop stability of the system is determined by the $\ell_\infty$ stability of the operator $(I-\Sp\tf A)^{-1}$.

Note that if $\tf A$ and $\tf B$ are memoryless and time invariant, i.e., if their matrix representations are block-diagonal $\tf A = \mathrm{blkdiag}(A,A,\dots)$, $\tf B = \mathrm{blkdiag}(B,B,\dots)$, then the system \eqref{eq:dynamics} reduce to the familiar finite-dimensional LTI system
\begin{equation}
x_{t+1} = A x_t + Bu_t + w_t, \, x_0 = 0,
\label{eq:lti-dynamics}
\end{equation}
where once again stability of the open-loop system is determined by the stability of the operator $(zI-A)^{-1}$ as opposed to the boundedness of the matrix $A$.

For LTI systems \eqref{eq:lti-dynamics}, the System Level Synthesis (SLS) framework \cite{wang2019system,anderson2019system} provides an appealing parameterization of all closed loop responses from $\tf w \to (\tf x, \tf u)$ achievable by a causal state-feedback $\ell_\infty$-stabilizing (equivalently internally stabilizing) LTI control law $\tf K$ such that $\tf u = \tf K \tf x$, as summarized in the following theorem.  We remind the reader that a controller $\tf K$ is $\ell_\infty$-stabilizing (equivalently internally stabilizing) if its interconnection with the system dynamics, as illustrated in Fig. \ref{fig:lti-interconnect}, defines an $\ell_\infty$-stable map from $(\tf w, \tf \delta_y, \tf \delta_u) \to (\tf x, \tf u, \what)$ (see III.A of \cite{wang2019system} ).

\begin{figure}
\centering
\includegraphics[width=.4\textwidth]{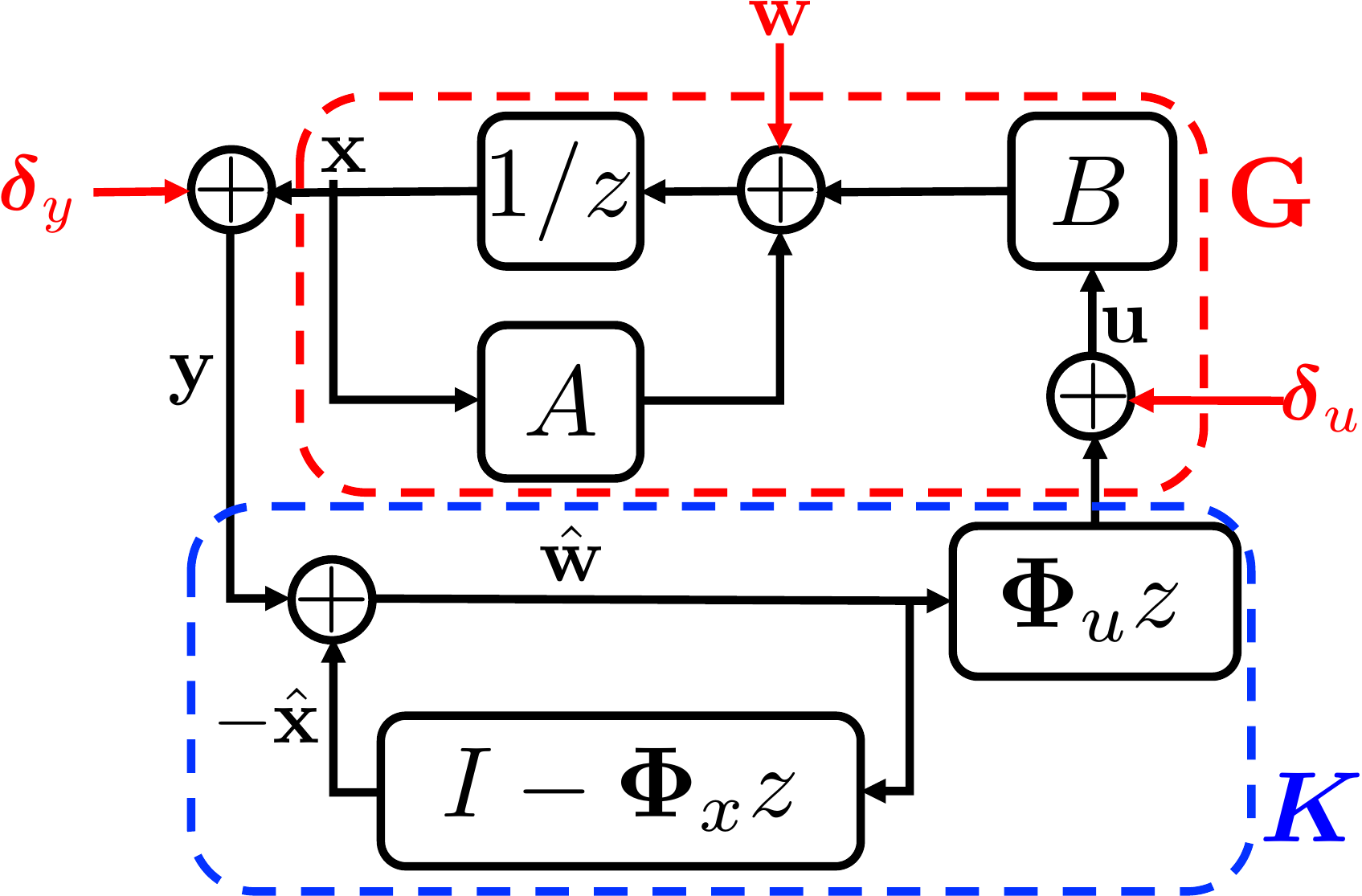}
\caption{The proposed $\ell_\infty$-stabilizing (equivalently internally stabilizing) state-feedback controller structure defined by equations \eqref{eq:lti-realization}.}
\label{fig:lti-interconnect}
\end{figure}

\begin{theorem}[Theorem 1, \cite{wang2019system}]\label{thm:lti-sls}
For the LTI system \eqref{eq:lti-dynamics} with causal state-feedback LTI control law $\tf u = \tf K \tf x$, the follwing are true:
\begin{enumerate}
\item The affine subspace defined by 
\begin{equation}
\begin{bmatrix} zI - A & - B \end{bmatrix} \begin{bmatrix} \Phix \\ \Phiu \end{bmatrix} = I, \, \ \Phix, \Phiu \in \frac{1}{z}\RHinf
\label{eq:lti-achievable}
\end{equation}
parameterizes all system responses
\begin{equation}
\begin{bmatrix} \tf x \\ \tf u \end{bmatrix} = \begin{bmatrix} \Phix \\ \Phiu \end{bmatrix} \tf w
\label{eq:lti-response}
\end{equation}
achievable by an internally stabilizing state-feedback controller $\tf K$.
\item For any transfer matrices $\left\{\Phix, \Phiu\right\}$ satisfying the constraints \eqref{eq:lti-achievable}, the control signal computed via\footnote{We note that due to the affine constraints \eqref{eq:lti-achievable}, $z\Phix-I$ is strictly causal, and thus feedback loop between $\tf{\hat x}$ and $\tf{\hat w}$ is well posed.}
\begin{equation}
\begin{array}{rcl}
\tf u &=& z\Phiu \tf{\hat w} \\
\tf{\hat w} &=& \tf x - \tf{\hat x}, \, \hat{w}_0 = 0, \\ 
\tf{\hat x} &=& (z\Phix - I)\tf{\hat w}
\end{array}
\label{eq:lti-realization}
\end{equation}
defines the control law $\uu = \Phiu\Phix^{-1}\x$, is internally stabilizing, and achieves the desired response \eqref{eq:lti-response}.
\end{enumerate}
\end{theorem}

Thus, in the case of LTI systems \eqref{eq:lti-dynamics}, the search for an optimal controller $\tf K$ can be converted to a search over system responses $\left\{\Phix,\Phiu\right\}$ constrained to lie in the affine space defined by equation \eqref{eq:lti-achievable}.  This fact, combined with the transparent mapping between the system responses and the corresponding controller implementation \eqref{eq:lti-realization}, has been successfully exploited for the synthesis of distributed optimal controllers for large-scale systems by imposing additional convex structural constraints, such as delay, sparsity, and locality subspace constriants, on the system responses and corresponding controller implementation \eqref{eq:lti-realization} -- we refer the reader to \cite{wang2014localized,wang2016localized,wang2018separable} for more details.

Another favorable feature of the parameterization defined in Theorem \ref{thm:lti-sls} is that it is provably stable under deviations from the subspace \eqref{eq:lti-achievable}, as summarized in the following theorem from \cite{matni2017scalable}.

\begin{theorem}[Theorem 2, \cite{matni2017scalable}]\label{thm:lti-robust}
Let $(\Phixh,\Phiuh,\D)$ be a solution to
\begin{equation}
\begin{bmatrix} zI - A & - B \end{bmatrix} \begin{bmatrix} \Phixh \\ \Phiuh \end{bmatrix} = I - \D, \, \ \Phixh, \Phiuh \in \frac{1}{z}\RHinf.
\label{eq:lti-robust}
\end{equation}
If $(I-\D)^{-1}$ exists as an operator from $\ell^n_{\infty,e} \to \ell^n_{\infty,e}$, then the controller implementation \eqref{eq:lti-realization} defined in terms of the transfer matrices $\left\{\Phixh,\Phiuh\right\}$ achieves the closed loop responses
\begin{equation}
\begin{bmatrix}\tf x \\ \tf u \end{bmatrix} = \begin{bmatrix} \Phixh \\ \Phiuh \end{bmatrix}(I-\D)^{-1}\tf w,
\end{equation}
on the LTI system \eqref{eq:lti-dynamics}, and is internally stabilizing if and only if $(I-\D)^{-1} \in \RHinf$.
\end{theorem}
This parameterization has proved crucial in providing tractable approximations to non-convex distributed control problems \cite{matni2017scalable}, and in providing sub-optimality bounds for robust controllers as applied to learned systems \cite{dean2017sample,dean2019safely}.  However, in these past works, very crude approximations based solely on small gain bounds and triangle inequalities were used to control the effects of the uncertain map $(I-\D)^{-1}$ on system stability and performance.  In this work we show that Theorems \ref{thm:lti-sls} and \ref{thm:lti-robust} can be extended to a more general setting that allows connections to well developed tools from the robust control literature \cite{khammash1990stability,dahleh1994control}.  Although we focus on $\mathcal{L}_1$ optimal control in this paper due to its favorable separability structure (cf. \S \ref{sec:extensions}), we expect these results to carry over naturally to the $\mathcal{H}_\infty$ setting.

\subsection{Necessary Conditions}
Here we characterize a set of affine constraints that the closed loop system responses of system \eqref{eq:dynamics} must satisfy if they are induced by a linear, causal, and $\ell_\infty$-stabilizing controller $\tf K : \ell^n_{\infty,e} \to \ell^p_{\infty,e}$ via the control law $\tf u = \tf K \tf x$.  In particular, a controller $\tf K$ is $\ell_\infty$-stabilizing if the interconnection illustrated in Fig. \ref{fig:simple-interconnect} is $\ell_\infty$-stable as a map from $(\tf w, \tf \delta_y, \tf \delta_u)\to (\x,\uu,\tf \beta)$, where $\tf \beta$ is any signal internal to the controller $\tf K$.

\begin{figure}
\centering
\includegraphics[width=.4\textwidth]{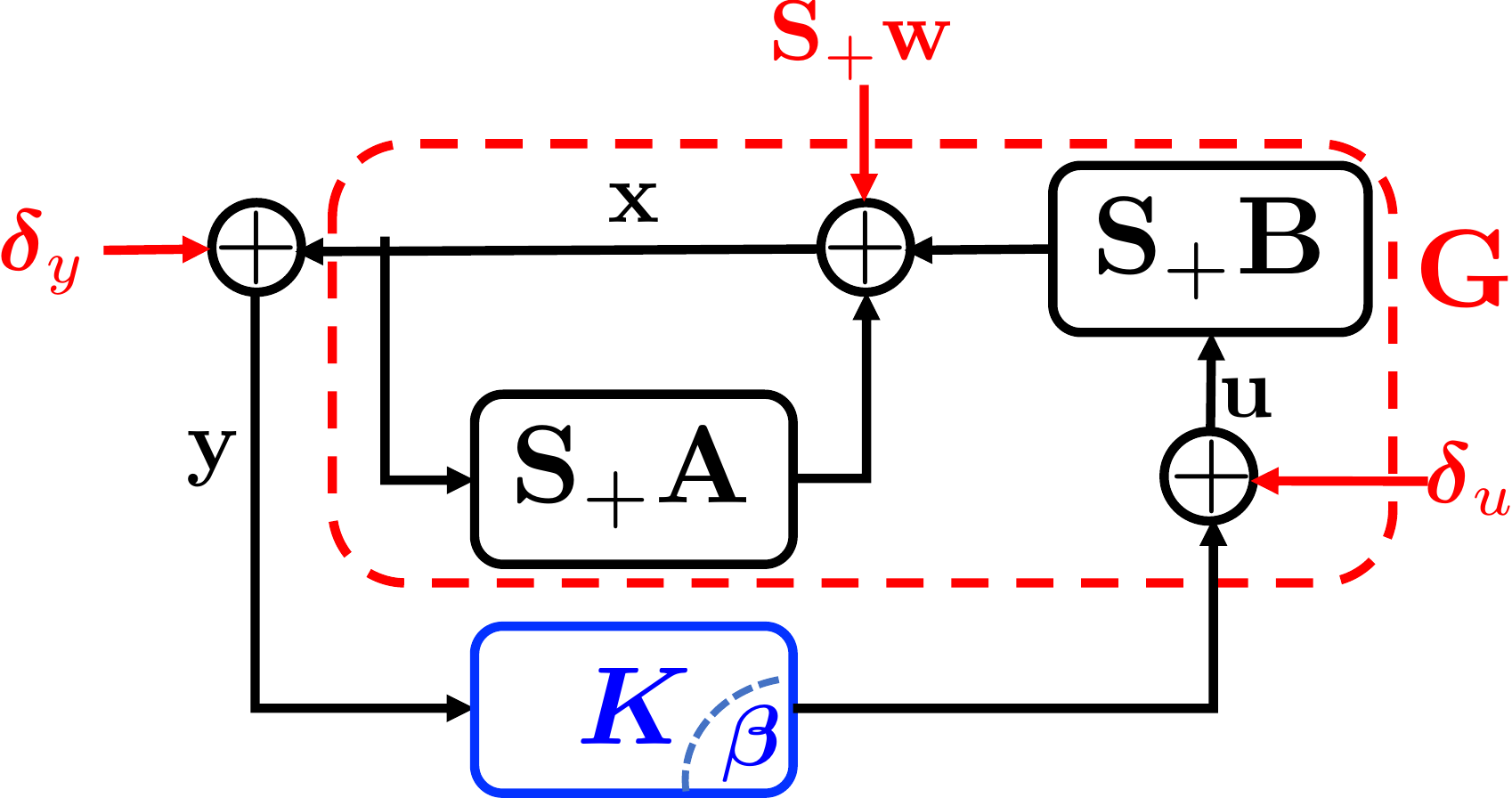}
\caption{A controller $\tf K$ is said to be $\ell_\infty$-stabilizing with respect to dynamics \eqref{eq:dynamics} if the illustrated interconnection is $\ell_\infty$-stable as a map from $(\tf w, \tf \delta_y, \tf \delta_u)\to (\x,\uu,\tf \beta)$, where $\tf \beta$ is any signal internal to the controller $\tf K$.}
\label{fig:simple-interconnect}
\end{figure}

\begin{proposition}\label{prop:necessity}
Let $\tf K: \ell^n_{\infty,e} \to \ell^p_{\infty,e}$ be a linear, causal, and $\ell_\infty$-stabilizing controller. 
%such that $(I-\Sp(\A+\B\K))^{-1}\Sp$ and $\K(I-\Sp(\A+\B\K))^{-1}\Sp$ are $\ell_\infty$-stable, i.e., let $\tf K$ be a linear causal controller such that closed loop maps from $\tf w \to (\tf x, \tf u)$ are $\ell_\infty$-stable.  
Then all maps taking $\tf w \to (\tf x, \tf u)$ achievable by such a $\tf K$ satisfy the constraints
\begin{equation}
\begin{aligned}
&\begin{bmatrix} I-\Sp\A & - \Sp\B\end{bmatrix}\begin{bmatrix} \Phix \\ \Phiu \end{bmatrix} = \Sp, \\ &\Phix, \Phiu \text{ strictly causal, linear, and $\ell_\infty$-stable.}
\end{aligned}
\label{eq:osls-achievable}
\end{equation}
\end{proposition}

\begin{proof}
Basic algebra shows that the map from $\tf w \to \x$ is given by $(I-\Sp(\A+\B\K))^{-1}\Sp$, and that the map from $\tf w \to \uu$ is given by $\K(I-\Sp(\A+\B\K))^{-1}\Sp$.  By assumption, both of these maps are $\ell_\infty$-stable,\footnote{That they are well posed is immediate as $\Sp(\A+\B\K)$ is strictly causal.} as $\tf K$ is an $\ell_\infty$-stabilizing controller.
%As $\Sp(\A+\B\K)$ is strictly causal, the feedback interconnection is well posed in the sense that $(I-\Sp(\A+\B\K))^{-1}$ exists as a map from $\ell^n_{\infty,e} \to \ell^n_{\infty,e}$.  
%By assumption, both $(I-\Sp(\A+\B\K))^{-1}$ and $\tf K(I-\Sp(\A+\B\K))^{-1}$ are $\ell_\infty$-stable. 
Defining
\[
\begin{bmatrix}
\Phix \\ \Phiu
\end{bmatrix} =
\begin{bmatrix} I \\ \K \end{bmatrix}(I-\Sp(\A+\B\K))^{-1}\Sp,
\]
it is easily verified that $\left\{\Phix,\Phiu\right\}$ satisfy constraint \eqref{eq:osls-achievable}.
\end{proof}
\begin{remark}
If $\A$ and $\B$ are LTI, and $\K$ is LTI, then so are $\{\Phix,\Phiu\}$, and all are uniquely characterized by their $z$-transforms.  Further the right shift operator $\Sp$, when restricted to $\LTI$, can be written as $\Sp=1/z$.  In this case, constraint \eqref{eq:osls-achievable} simplifies to
\begin{equation}
\begin{bmatrix} I - \frac{1}{z}\A& -\frac{1}{z}\B \end{bmatrix} \begin{bmatrix} \Phix \\ \Phiu\end{bmatrix} = \frac{1}{z}I, \, \ \Phix, \Phiu \in \frac{1}{z}\RHinf,
\end{equation}
which can be viewed as a generalization of constraint \eqref{eq:lti-achievable} to LTI dynamics defined by dynamic operators $\A(z)$ and $\B(z)$.  Similarly, if $\A$ and $\B$ are memoryless and LTI, and $\K$ is LTI, then $\{\Phix,\Phiu\}$ are LTI, and constraint \eqref{eq:osls-achievable} simplifies to
\begin{equation}
\begin{bmatrix} I - \frac{1}{z}A & -\frac{1}{z}B \end{bmatrix} \begin{bmatrix} \Phix \\ \Phiu \end{bmatrix} = \frac{1}{z}I, \, \ \Phix, \Phiu \in \frac{1}{z}\RHinf,
\end{equation}
which is clearly equivalent to \eqref{eq:lti-achievable}.
\end{remark}

\subsection{Controller Implementation}
\begin{figure}
\centering
\includegraphics[width=.4\textwidth]{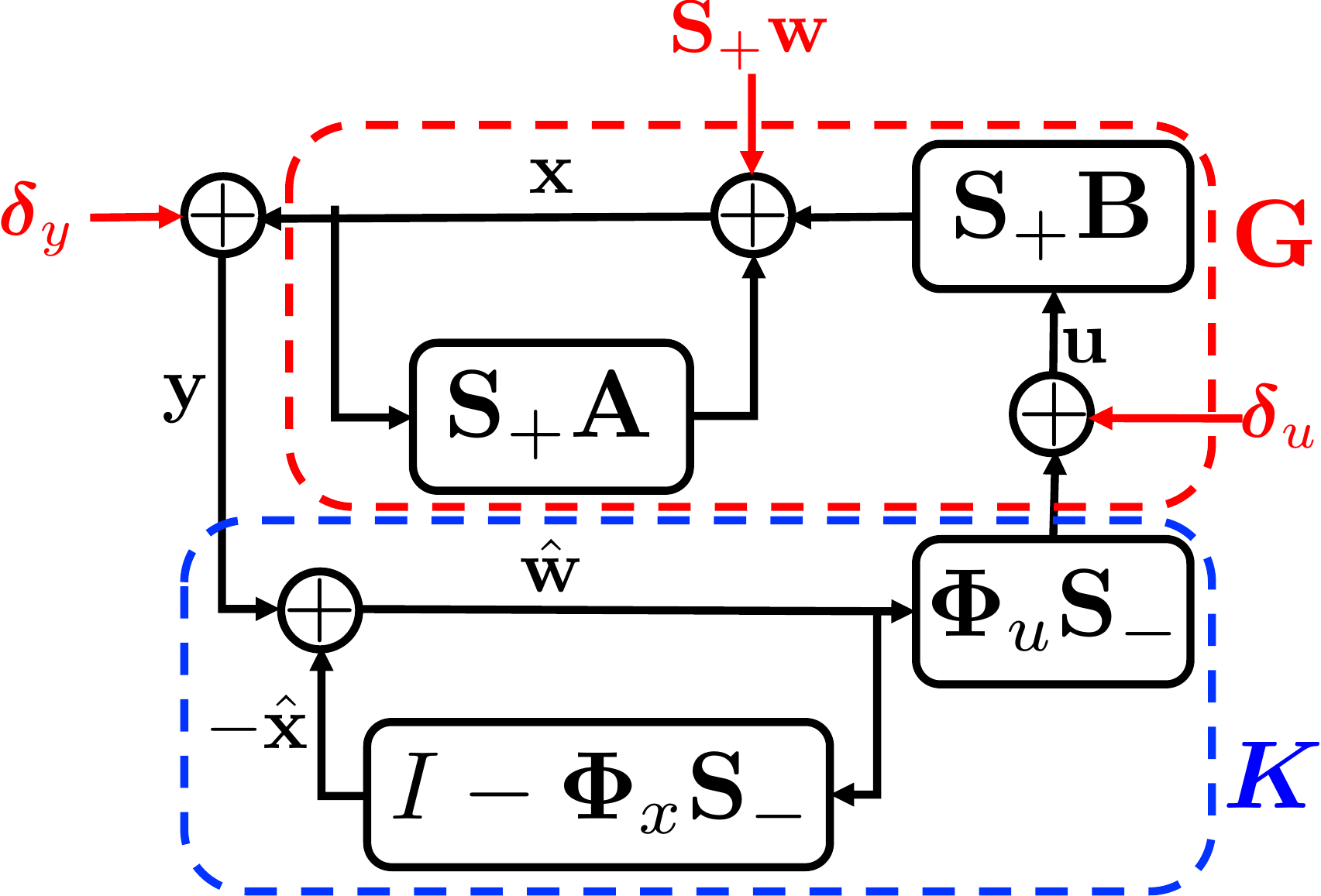}
\caption{The proposed state-feedback controller structure defined by equations \eqref{eq:realization}.}
\label{fig:realization}
\end{figure}
We now show how to construct an internally stabilizing controller from any strictly causal, linear, and $\ell_\infty$-stable operators $\left \{\Phix,\Phiu\right\}$ satisfying constraint \eqref{eq:osls-achievable} that achieves the desired response from $\tf w \to (\tf x, \tf u)$.
\begin{proposition}\label{prop:sufficiency}
Let $\left\{\Phix,\Phiu\right\}$ satisfy constraint \eqref{eq:osls-achievable}.  Then the controller implementation shown in Fig. \ref{fig:realization}, described by the equations
\begin{equation}
\begin{array}{rcl}
\tf u &=& \Phix \Sm \tf{\hat w}\\
\tf{\hat w} & = & \x - \tf{\hat x}, \, \hat{w}_0 = 0, \\
\tf{\hat x} &=& (\Phix\Sm-I)\tf{\hat w},
\end{array}
\label{eq:realization}
\end{equation}
is $\ell_\infty$-stabilizing.  In particular, the resulting map from $(\tf w, \tf \delta_y, \tf \delta_u) \to (\tf x, \tf u, \tf{\hat w})$ is $\ell_\infty$-stable, and achieves the desired closed loop response
\begin{equation}
\begin{bmatrix} \tf x \\ \tf u \end{bmatrix} = \begin{bmatrix} \Phix \\ \Phiu \end{bmatrix}\tf w.
\label{eq:response}
\end{equation}
\end{proposition}
\begin{proof}
From equations \eqref{eq:dynamics} and \eqref{eq:realization} (alternatively, from Fig. \ref{fig:realization}), we have that
\begin{equation}
\begin{array}{rcl}
\x &=& \Sp\A\x + \Sp\B\uu + \Sp \w, \ x_0 =0\\
\uu &=& \Phiu\Sm\what + \tf \delta_u\\
\what &=& \x + \tf \delta_y + (I-\Phix\Sm)\what, \ \hat{w}_0=0,
\end{array}
\label{eq:internal-stability}
\end{equation}
where we emphasize that $\hat{w}_0 = 0$ such that $\what$ is a strictly causal signal.  We first observe that the restriction of $(I-\Phix\Sm)$ to strictly causal signals is itself strictly causal, as the constraint \eqref{eq:osls-achievable} enforces that the block lower triangular matrix representation of $\Phix$ has identify matrices along its first block sub-diagonal, i.e., for $\Phix = (\Phi_x(i,j))_{i,j=0}^\infty$ the block lower triangular matrix representation of $\Phix$, we have that $\Phi_x(i,i-1)=I$ for all $i\geq 1$.  It therefore follows that the feedback loop between $\xhat$ and $\what$ is well posed.  By rote calculation, it follows from equation \eqref{eq:internal-stability} that the closed loop maps from $(\tf w, \tf \delta_y, \tf \delta_u) \to (\tf x, \tf u, \tf{\hat w})$ are given by
\begin{equation}
\begin{bmatrix}
\x \\ \uu \\ \what
\end{bmatrix} =
\begin{bmatrix} \Phix & \Phix(\Sm - \A) & \Phix\B \\
\Phiu & \Phiu(\Sm-\A) & I + \Phiu\B \\
\Sp & I - \Sp\A & \Sp \B
\end{bmatrix} \begin{bmatrix} \tf w \\ \tf \delta_y \\ \tf \delta_u \end{bmatrix}.
\end{equation}
By assumption, $\Phix$, $\Phiu$, $\A$, and $\B$ are all $\ell_\infty$-stable, and hence the interconnection illustrated in Fig. \ref{fig:realization}, and described by equations \eqref{eq:dynamics} and \eqref{eq:internal-stability} is $\ell_\infty$-stable.
\end{proof}

\begin{remark}
If $\A$ and $\B$ are memoryless and LTI, and $\K$ is LTI, then so are the system responses $\left\{\Phix,\Phiu\right\}$, and consequently the right and left shift operators $\Sp$ and $\Sm$ become $\frac{1}{z}$ and $z$, respectively, recovering the controller implementation \eqref{eq:lti-realization}.
\end{remark}

\subsection{Robust Operator System Level Synthesis}
Thus Propositions \ref{prop:necessity} and \ref{prop:sufficiency} show that the parameterization of Theorem \ref{thm:lti-sls} can be extended to a class of dynamics described by bounded and causal linear operators in feedback with a causal linear controller.  We now show that this extension enjoys similar stability properties with respect to perturbations from the subspace \eqref{eq:osls-achievable}.

\begin{theorem}\label{thm:robust-sls}
Let $\A\in\LTV^{n,n}$ and $\B\in\LTV^{n,p}$, and suppose that $\left\{\Phixh,\Phiuh\right\}$ satisfy
\begin{equation}
\begin{aligned}
&\begin{bmatrix} I-\Sp\A & - \Sp\B\end{bmatrix}\begin{bmatrix} \Phixh \\ \Phiuh \end{bmatrix} = \Sp(I-\D), \\ 
& \Phix, \Phiu \text{ strictly causal, linear, and $\ell_\infty$-stable,}
\end{aligned}
\end{equation}
for $\D$ a strictly causal linear operator from $\ell^n_{\infty,e}\to \ell^n_{\infty,e}$.  Then the controller implementation \eqref{eq:realization} defined in terms of the operators $\left\{\Phixh,\Phiuh\right\}$ is well posed and achieves the following response
\begin{equation}
\begin{bmatrix} \x \\ \uu \end{bmatrix} = \begin{bmatrix}\Phixh\\ \Phiuh \end{bmatrix}(I-\D)^{-1}\w.
\label{eq:robust-response}
\end{equation}
Further, this interconnection is $\ell_\infty$-stable if and only if $(I-\D)^{-1}$ is $\ell_\infty$-stable.
\end{theorem}
\begin{proof}
As $\D$ is strictly causal by assumption, $I_\D:= (I-\D)^{-1}$ exists as a map from $\ell^n_{\infty,e} \to \ell^n_{\infty,e}$.  Going through a similar argument as that in the proof of Proposition \ref{prop:sufficiency}, we observe that
\begin{equation}
\scriptstyle
\begin{bmatrix}
\x \\ \uu \\ \what
\end{bmatrix} =
\begin{bmatrix} \Phix I_\D & \Phix I_\D(\Sm - \A) & \Phix I_\D \B \\
\Phiu I_\D& \Phiu I_\D(\Sm-\A) & I + \Phiu I_\D\B \\
\Sp I_\D & I_\D(I - \Sp\A) & I_\D\Sp \B
\end{bmatrix} \begin{bmatrix} \tf w \\ \tf \delta_y \\ \tf \delta_u \end{bmatrix}.
\end{equation}

Thus we see that the desired map \eqref{eq:robust-response} from $\w \to (\x,\uu)$ is achieved.  Further, as $\Phix$, $\Phiu$, $\A$, $\B$ are all $\ell_\infty$-stable by assumption, it follows that the $\ell_\infty$-stability of the map from $(\tf w, \tf \delta_y, \tf \delta_u) \to (\x, \uu, \what)$ is determined by the $\ell_\infty$-stability of $I_\D$, from which the result follows.
\end{proof}

\section{Robust Performance under Model Uncertainty}
\label{sec:robust-perf}
%!TEX root = main.tex
We now use the tools developed in the previous section to identify necessary and sufficient conditions for the robust stability and robust performance of a system \eqref{eq:dynamics} subject to bounded perturbations in its $\A$ and $\B$ operators.  In particular consider the system
\begin{equation}
\tf x = \Sp(\A_0 + \DA)\x + \Sp(\B_0 + \DB)\uu + \Sp \w,
\label{eq:uncertain-dynamics}
\end{equation}
where $\A_0 = \mathrm{blkdiag}(\hat A,\hat A,\dots)$ and $\B_0 = \mathrm{blkdiag}(\hat B, \hat B,\dots)$ are memoryless LTI operators defining a nominal LTI system $x_{t+1} = \hat Ax_t + \hat Bu_t + w_t$, and $\DA$ and $\DB$ are $\ell_\infty$-stable and satisfy 
\begin{equation}
\linffnorm{[\DA,\, \DB]}\leq \eps.
\label{eq:eps-bound}
\end{equation}
We first identify SLS based necessary and sufficient conditions for robust stability, and then build upon those to formulate a robust performance problem.  

We consider the following robust control problem: find a LTI controller $\tf K : \ell^n_{\infty,e} \to \ell^p_{\infty,e}$, using only the nominal dynamics $(\hat A,\hat B)$ and uncertainty bound $\eps$, such that the dynamics \eqref{eq:uncertain-dynamics} in closed loop with the control law $\uu = \K \x$ is $\ell_\infty$-stable for all admissible uncertainty realizations $(\DA,\DB)$ satisfying \eqref{eq:eps-bound}.  To do so, we recognize that for any LTI $\{\Phixh,\Phiuh\}$ satisfying the LTI achievability constraints \eqref{eq:lti-achievable} defined by $(\hat A, \hat B)$, we then have that
\begin{multline}
\begin{bmatrix} I - \Sp\A_0 - \Sp\DA & - \Sp\B_0 - \Sp\DB \end{bmatrix}\begin{bmatrix} \Phixh \\ \Phiuh \end{bmatrix} \\= \begin{bmatrix} I - \Sp\A_0 & - \Sp\B_0 \end{bmatrix}\begin{bmatrix} \Phixh \\ \Phiuh \end{bmatrix} - \begin{bmatrix} \Sp\DA & \Sp\DB \end{bmatrix}\begin{bmatrix} \Phixh \\ \Phiuh \end{bmatrix}\\ = \Sp\left(I - \Sp\begin{bmatrix} \DA & \DB \end{bmatrix}\begin{bmatrix} \Phixh \\ \Phiuh \end{bmatrix} \right)
\end{multline}
where the final inequality follows from the assumption that $\{\Phixh,\Phiuh\}$ satisfy the LTI achievability constraints \eqref{eq:lti-achievable}, or equivalently \eqref{eq:osls-achievable}, defined in terms of the dynamics $(\hat A, \hat B)$.  Noting that
\begin{equation}
\Dh := \Sp\begin{bmatrix} \DA & \DB \end{bmatrix}\begin{bmatrix} \Phixh \\ \Phiuh \end{bmatrix},
\end{equation}
is a strictly causal $\ell_\infty$-stable operator, we conclude by Theorem \ref{thm:robust-sls} that the controller implementation \eqref{eq:realization} defined in terms of the LTI operators $\{\Phixh,\Phiuh\}$ achieves the following closed loop behavior when applied to the uncertain dynamics \eqref{eq:uncertain-dynamics}:
\begin{equation}
\begin{bmatrix}\x \\ \uu \end{bmatrix} = \begin{bmatrix} \Phixh \\ \Phiuh \end{bmatrix}(I-\Dh)^{-1}.
\label{eq:dhat-response}
\end{equation}
Further this control law is internally stabilizing if and only if $(I-\Dh)^{-1}$ is $\ell_\infty$-stable.  

Defining the controlled output signal as 
\begin{equation}
\tf z = \tf C\x + \tf D\uu,
\label{eq:z-output}
\end{equation}
for $\tf C = \mathrm{blkdiag}(C,C,\dots)$ and $\tf D = \mathrm{blkdiag}(D,D,\dots)$ user specified cost matrices,\footnote{For simplicity, we assume $\tf C$ and $\tf D$ to be memoryless and LTI, however our results are equally applicable when the controlled output is defined in terms of LTI  filters $\tf C(z)$ and $\tf D(z)$.} and consider the goal of minimizing the $\ell_\infty\to\ell_\infty$ induced gain from $\tf w \to \tf z$ of the uncertain system \eqref{eq:uncertain-dynamics}.  We can then pose the robust performance problem for a specified performance level $\gamma \geq 0$ as finding LTI operators $\{\Phixh,\Phiuh\}$ that satisfy 
\begin{equation}\label{eq:robust-perf1}
\begin{aligned}
& \linffnorm{\begin{bmatrix} \tf C & \tf D \end{bmatrix} \begin{bmatrix} \Phixh \\ \Phiuh \end{bmatrix} \left(I-\Sp\begin{bmatrix} \DA & \DB \end{bmatrix}\begin{bmatrix} \Phixh \\ \Phiuh \end{bmatrix}\right)^{-1} } \leq \gamma \\
& \begin{bmatrix} zI - \hat A   - \hat B \end{bmatrix}\begin{bmatrix} \Phixh \\ \Phiuh \end{bmatrix} = I,\ \Phixh, \Phiuh \in \frac{1}{z}\RHinf, \\
& \left(I-\Sp\begin{bmatrix} \DA & \DB \end{bmatrix}\begin{bmatrix} \Phixh \\ \Phiuh \end{bmatrix}\right)^{-1} \text{ is $\ell_\infty$-stable}
\end{aligned}
\end{equation}
for all $(\DA,\DB)$ satisfying bound \eqref{eq:eps-bound}, where we have combined equations \eqref{eq:z-output} and \eqref{eq:dhat-response} to derive the robust performance bound condition.

To lighten notation going forward, we let
\begin{equation}
\tf Q := \begin{bmatrix} \tf C & \tf D \end{bmatrix}, \ \Phih := \begin{bmatrix}\Phixh \\ \Phiuh \end{bmatrix}, \ \D := \frac{1}{\eps}\Sp\begin{bmatrix} \DA & \DB \end{bmatrix}, \ Z_{AB} := \begin{bmatrix} zI-\hat A & - \hat B\end{bmatrix}.
\end{equation}

With this notation, the robust performance problem \eqref{eq:robust-perf1} is equivalent to finding an LTI operator $\Phih$ satisfying
\begin{equation} \label{eq:robust-perf}
\begin{aligned}
& \linffnorm{\frac{1}{\gamma}\Q\Phih + \frac{1}{\gamma}\Q\Phih\D(I-(\eps\Phih)\D)^{-1}(\eps\Phih)} \leq 1 \\
& Z_{AB}\Phih = I,\ \Phih \in \frac{1}{z}\RHinf,\   (I-\Phih\D)^{-1} \text{ is $\ell_\infty$-stable}\\
\end{aligned}
\end{equation}
for all $\D$ satisfying $\linffnorm{\D}\leq 1$, where we have used that $(I-\Delta\Phih)^{-1} = I + \Delta(I-\Phih\Delta)^{-1}\Phih$ and that $(I-GH)^{-1}$ is $\ell_\infty$-stable if and only if $(I-HG)^{-1}$ is $\ell_\infty$-stable (see Proposition 1, \cite{khammash1990stability}) to recast the expression \eqref{eq:robust-perf1} in a form that matches the linear-fractional-transform (LFT) structure studied in \cite{khammash1990stability,dahleh1994control}.

We can therefore leverage the equivalence between robust stability and performance (see Theorem 5.1, \cite{khammash1990stability}) to conclude that $\Phih$ satisfies the robust performance conditions \eqref{eq:robust-perf} for all $\linffnorm{\D}\leq 1$ if and only if the augmented LTI system
\begin{equation}
\tf M = \begin{bmatrix} \frac{1}{\gamma}\Q\Phih & \frac{1}{\gamma}\Q\Phih \\ \eps\Phih & \eps\Phih \end{bmatrix}
\label{eq:bigM}
\end{equation}
is robustly stable for all structured perturbations $\tilde{\D}$ satisfying
\begin{equation}
\tilde{\D} = \mathrm{blkdiag}\left(\D_1,\D_2\right), \, \D_1, \D_2 \in \LTV, \,  \linffnorm{\D_1} \leq 1, \linffnorm{\D_2} \leq 1.
\end{equation}

The necessary and sufficient conditions for robust stability of the resulting two-block problem can be derived as a special case of Theorem 6.3 of \cite{khammash1990stability}.  The particular case of an  augmented LTI system $\tf M$ satisfying $\tf M_{11} = \tf M_{12}$ and $\tf M_{21} = \tf M_{22}$, as is the case for our problem \eqref{eq:bigM}, is addressed in Ch 8.3 of \cite{khammash1990stability}, where a similarly structured augmented system \eqref{eq:bigM} arises in the context of bounding output sensitivity in the presence of output perturbations.  The necessary and sufficient conditions specified in Theorem 6.3 of \cite{khammash1990stability} reduce to the following \emph{convex} constraints on the system response $\Phih$
\begin{equation} \label{eq:constraints}
\begin{aligned}
&Z_{AB}\Phih = I, \, \Phih \in \frac{1}{z}\RHinf, \\
& \linffnorm{\Q\Phih} + \gamma\linffnorm{\eps\Phih} < \gamma.
\end{aligned}
\end{equation}

Finally, we remark that although the constraints \eqref{eq:constraints} are in general infinite-dimensional due to the transfer matrix $\Phih$, principled finite-dimensional approximations, some of which enjoy provable sub-optimality guarantees, are available \cite{matni2017scalable,anderson2019system,dean2017sample,dean2018regret}.  Further, for the $\mathcal{L}_1$ problem considered here, the resulting optimization problem can be posed as a linear program, thus enjoying favorable computational complexity properties.
It then follows that by bisecting on $\gamma$, e.g., by using golden search, we can find a performance level $\gamma$, and corresponding system responses and controller, satisfying $\gamma \leq \gamma_\star + \epsilon$ in $O\log_2(1/\epsilon)$ iterations, for $\gamma_\star$ the smallest $\gamma$ such that the set defined by \eqref{eq:constraints} is non-empty.
%Thus, an LTI controller $\K=\Phiuh\Phixh^{-1}$ is robustly stabilizing for system \eqref{eq:uncertain-dynamics} if and only if the operators $\left\{\Phixh,\Phiuh\right\}$ satisfy

\section{Large-Scale Distributed Control}
\label{sec:extensions}
%!TEX root = main.tex
%In this section, we discuss to extensions to the robust performance results presented in the previous section.

\subsection{Robust Performance Guarantees for Large-Scale Distributed Control}
\label{sec:extensions}
In previous work \cite{wang2014localized,wang2016localized,wang2018separable}, it was shown that for LTI dynamical systems \eqref{eq:lti-dynamics} defined by structured (i.e., sparse) matrices $(A,B)$, imposing \emph{locality constraints} on the system responses $\{\Phix,\Phiu\}$, i.e., imposing that $\{\Phix,\Phiu\} \in \mathcal{S}$, for $\mathcal{S}$ a suitably defined structure inducing subspace constraint, leads to distributed controllers that enjoyed scalable synthesis and implementation complexity.  Although formally defining these concepts is beyond the scope of this paper, we note that such conditions can be easily enforced on the solution of the robust performance conditions \eqref{eq:constraints} by additionally imposing that $\{\Phixh,\Phiuh\}\in\mathcal{S}$.   Under suitable assumptions on the structure of the cost matrices $(C,D)$ (e.g., that $(C,D)$ are block-diagonal), the resulting problem satisfies a notion of \emph{partial separability}, c.f. \S IV of \cite{wang2018separable}, which allows for the problem to be solved at scale using tools from distributed optimization.  

We emphasize that the conditions identified in \eqref{eq:constraints} remain necessary and sufficient when additional structure is imposed on the system responses $\{\Phixh,\Phiuh\}$ so long as the dynamic perturbations $(\DA,\DB)$ remain unstructured.  In particular, Theorem 6.3 of \cite{khammash1990stability} is applicable to any augmented plant $\tf M \in \RHinf$ -- this condition holds true for the augmented plant defined in equation \eqref{eq:bigM} even when any additional constraints are imposed on $\Phih$.  To the best of our knowledge, these are the first such necessary and sufficient conditions for robust performance that are applicable to large-scale distributed systems.  An exciting direction for future work will be to explore the consequences of locality in the system responses $\{\Phixh,\Phiuh\}$ on necessary and sufficient conditions for robust performance when the perturbations $(\DA,\DB)$ are further constrained to respect the topology of the nominal system, as defined by the support of $(\hat A, \hat B)$.
%
%\subsection{Sub-Optimality Guarantees for Robust Control}
%In \cite{dean2017sample}, it was shown\footnote{The result in \cite{dean2018sample} is for a robust $\mathcal{H}_2$ problem, but the derivation goes through nearly as is for $\mathcal{L}_1$ problem considered here.} that an upper bound to the robust performance problem considered here can be computed by solving
%\begin{equation}
%\begin{array}{rl}
%\min_{\tau \in [0,1)}&\frac{1}{1-\tau}\min_{\Phix,\Phiu}\linffnorm{\begin{bmatrix} C & D \end{bmatrix} \begin{bmatrix} \Phixh \\ \Phiuh \end{bmatrix}} \\
%\text{subject to} & \begin{bmatrix} zI-\hat A & - \hat B \end{bmatrix}\begin{bmatrix}\Phixh \\ \Phiuh \end{bmatrix} = I, \ \linffnorm{\begin{bmatrix} \Phixh \\ \Phiuh \end{bmatrix}} \leq \frac{\tau}{\epsilon}, \\
%&\Phixh, \Phiuh \in \frac{1}{z}\RHinf,
%\end{array}
%\end{equation}
%and further, it was shown that $\epsilon$ is such that $\epsilon(1+\linffnorm{\K_\star})\linffnorm{(zI-(A+B\K_\star))^{-1}}\leq 1/5$, for $\K_\star$ the optimal co

\section{Experiments}
\label{sec:experiments}
%!TEX root = main.tex
All code needed to reproduce examples in this section can be found at \url{https://github.com/unstable-zeros/robust-sls}. We consider a scaled doubly-stochastic chain system described by the following dynamics:
\begin{equation}\label{eq:chain}
\begin{array}{rcl}
x^1_{t+1} &=& \rho\left[(1-\alpha)x^1_t + \alpha x^2_t\right] + u^1_t\\
x^i_{t+1} &=& \rho\left[\alpha x^{i-1}_t + (1-2\alpha)x^i_t + \alpha x^{i+1}_t\right] + u^i_t,\, \text{ for $i=2,\dots,N-1$,} \\
 x^N_{t+1} &=& \rho\left[\alpha x^{N-1}_t + (1-\alpha)x^N_t\right] + u^N_t
 \end{array}
\end{equation}
where the $x^i_t, u^i_t \in \R$ are the scalar state and inputs, respectively, of the subsystems, and we set the number of scalar subsystems $N=50$, the scaling factor $\rho = 0.5$, and the coupling constant $\alpha = 0.49$.  

We solve the robust performance performance problem \eqref{eq:constraints} under both centralized and localized distributed constraints with a norm bound on the uncertainty of $\epsilon = 0.55$, and cost matrices $C^\top = [I_N, 0^\top]^\top$ \& $D^\top = [0^\top, 5I_N]$.  For both settings, we impose an FIR horizon of $T=10$ when solving the robust performance problem \eqref{eq:constraints}.  Additionally, we enforce that the corresponding system responses satisfy $d$-locality constraints -- intuitively, these constraints ensure that in closed loop, the disturbance striking node $i$ only affects nodes $j$ satisfying $|j-i|\leq d$.\footnote{In the interest of clarity, we do not enforce communication delay constraints, but note that both communication delay and locality constraints can be enforced through suitable sparsity constraints on the system response variables: see \cite{anderson2019system} for details.}

By bisecting on $\gamma$, we determine that the optimal robust performance level is $\gamma = 5.57$ for both centralized and distributed controllers, where for the distributed localized controller we set the locality diameter to $d=2$.  That there is no gap between centralized and distributed is not surprising because: (i) we impose no communication delay constraints, and (ii) $\mathcal{L}_1$ optimal control leads to deadbeat optimal closed loop responses, which will consequently also be (approximately) localized in space as well.  We note that the nominal $\mathcal{L}_1$ norms of the closed loop systems for the centralized and distributed localized controllers are both $2.5$.  Comparing these to the norms achieved by the optimal $\mathcal{L}_1$ controllers (i.e., those computed by minimizing the performance cost with $\epsilon = 0$) of $1.43$ and $2.47$, respectively, we see that while there is an appreciable degradation in nominal performance in the centralized setting, there is nearly no degradation in the localized distributed setting!  We conjecture that this is due to the sparsity of the augmented plant $\tf M$ defined by the system response $\Phih$, which constrains both robust and nominal systems to behave similarly.

To empirically test this conjecture, we examine the evolution of the closed loop norm of the nominal and robust controllers to perturbations of the form $\DA = \mathrm{blkdiag}(\kappa I, \kappa I, \dots)$ and $\DB = 0$, for $\kappa \in [0,\epsilon]$, for varying locality parameters $d\in\{2,5,10\}$, where $d=10$ corresponds to the centralized setting.  The results are displayed in Fig \ref{fig:kappa}.  We show only the results for $d=2$ and $d=5$, as the result for $d=5$ and $d=10$ are indistinguishable -- as can be observed, in the ``extremely'' localized setting of $d=2$, the degrees of freedom are limited such that robust and nominal control behave similarly; in contrast, when $d=5$, the robust controller enjoys improved performance for larger values of $\kappa$, at the expense of degraded performance at lower values.  Our approach therefore allows for a principled exploration of tradeoffs between synthesis/implementation complexity (as measured by $d$), nominal performance, and robust performance for large-scale distributed systems.

\begin{figure}
\centering
\includegraphics[width=.45\columnwidth]{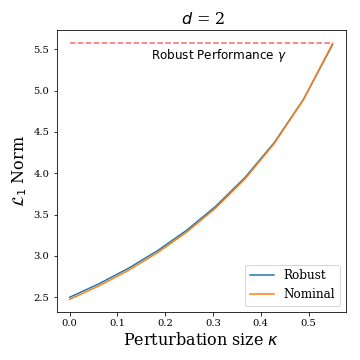}~~~~~\includegraphics[width=.45\columnwidth]{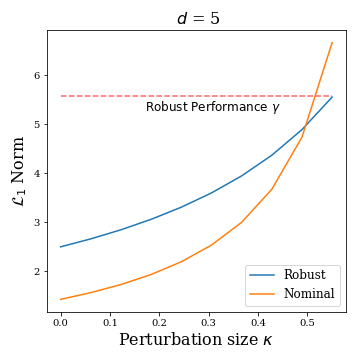}
\caption{Performance of a robust controller and a nominally optimal controller for for two decentralized chains, with disturbances  $\DA = \mathrm{blkdiag}(\kappa I, \kappa I, \dots)$. As $\kappa$ increases to $\epsilon = 0.55$, performance of the robust controller meets the robust performance bound $\gamma$ \eqref{eq:constraints}.}
\label{fig:kappa}
\end{figure}

%Finally we note that even without exploiting any of the underlying partial separability of the resulting problem (see \cite{wang2018separable}), we are able to solve the resulting centralized and distributed localized feasibility problems \eqref{eq:constraints} in 29s and 114s,\footnote{The majority of this time was spent on pre-processing by MOSEK: the actual solve time required by the optimizer was under 1s.} respectively, on a 2.3GHz 8 Core MacBook Pro with 32 GB of RAM using CVXPY \cite{diamond2016cvxpy} and MOSEK \cite{andersen2000mosek}.

\section{Conclusion and Future Work}
\label{sec:conclusion}
%!TEX root = main.tex
In this paper, we generalized the SLS parameterization of LTI $\ell_\infty$-stabilizing controllers, as well as a robust counterpart, to systems described by bounded and causal linear operators.  This extension, along with a simple algebraic transformation, allowed us to leverage tools from $\mathcal{L}_1$ robust control to derive necessary and sufficient conditions for the robust performance of an uncertain system in terms of convex constraints on the system response.  We argued that these conditions remain necessary and sufficient when additional structure, such as that induced by delay, sparsity, and locality constraints, are imposed on the system response, so long as the uncertain elements $(\DA,\DB)$ remained unstructured.  Further, in the case of $\mathcal{L}_1$ optimal control, the resulting robust performance criteria satisfy the partial-separability properties (assuming suitable structural constraints on the cost matrices) needed to apply the distributed synthesis techniques described in \cite{wang2018separable}, thus making our results applicable to large-scale distributed systems. 

More importantly, we believe that the results in this paper open up a wide and exciting range of future research directions, including but not limited to, deriving analogous results for $\mathcal{H}_\infty$ optimal control, revisiting the structured singular value, $\mu$-synthesis, and structured small gain theorems from a system level perspective, and perhaps most exciting, exploring the interplay between closed loop locality constraints and additional structure in the dynamic uncertainty $(\DA,\DB)$.  Further, it is of interest to see if these tighter conditions can be used to derive interpretable bounds on the degradation in performance of a robust controller as a function of the norm bound $\eps$ on the uncertainty $\linffnorm{[\DA,\DB]}$, as coarser and more conservative versions of such bounds have proved crucial in combining machine learning and robust control techniques \cite{dean2017sample,dean2018regret,dean2019safely}.

\begin{small}
\bibliographystyle{abbrvnat}  
\bibliography{references,Distributed} 

\begin{thebibliography}{25}
\providecommand{\natexlab}[1]{#1}
\providecommand{\url}[1]{\texttt{#1}}
\expandafter\ifx\csname urlstyle\endcsname\relax
  \providecommand{\doi}[1]{doi: #1}\else
  \providecommand{\doi}{doi: \begingroup \urlstyle{rm}\Url}\fi

\bibitem[{Ahmadi} et~al.(2018){Ahmadi}, {Cubuktepe}, {Topcu}, and
  {Tanaka}]{ahmadi2018distributed}
M.~{Ahmadi}, M.~{Cubuktepe}, U.~{Topcu}, and T.~{Tanaka}.
\newblock Distributed synthesis using accelerated admm.
\newblock In \emph{2018 Annual American Control Conference (ACC)}, pages
  6206--6211, June 2018.
\newblock \doi{10.23919/ACC.2018.8431859}.

\bibitem[Anderson et~al.(2011)Anderson, Teixeira, Sandberg, and
  Papachristodoulou]{anderson2011dynamical}
J.~Anderson, A.~Teixeira, H.~Sandberg, and A.~Papachristodoulou.
\newblock Dynamical system decomposition using dissipation inequalities.
\newblock In \emph{2011 50th IEEE Conference on Decision and Control and
  European Control Conference}, pages 211--216. IEEE, 2011.

\bibitem[Anderson et~al.(2019)Anderson, Doyle, Low, and
  Matni]{anderson2019system}
J.~Anderson, J.~C. Doyle, S.~H. Low, and N.~Matni.
\newblock System level synthesis.
\newblock \emph{Annual Reviews in Control}, 2019.

\bibitem[Arcak et~al.(2016)Arcak, Meissen, and Packard]{arcak2016networks}
M.~Arcak, C.~Meissen, and A.~Packard.
\newblock \emph{Networks of dissipative systems: compositional certification of
  stability, performance, and safety}.
\newblock Springer, 2016.

\bibitem[Dahleh and Diaz-Bobillo(1994)]{dahleh1994control}
M.~A. Dahleh and I.~J. Diaz-Bobillo.
\newblock \emph{Control of uncertain systems: a linear programming approach}.
\newblock Prentice-Hall, Inc., 1994.

\bibitem[Dean et~al.(2017)Dean, Mania, Matni, Recht, and Tu]{dean2017sample}
S.~Dean, H.~Mania, N.~Matni, B.~Recht, and S.~Tu.
\newblock On the sample complexity of the linear quadratic regulator.
\newblock \emph{arXiv preprint arXiv:1710.01688}, 2017.

\bibitem[Dean et~al.(2018)Dean, Mania, Matni, Recht, and Tu]{dean2018regret}
S.~Dean, H.~Mania, N.~Matni, B.~Recht, and S.~Tu.
\newblock Regret bounds for robust adaptive control of the linear quadratic
  regulator.
\newblock In \emph{Advances in Neural Information Processing Systems}, pages
  4188--4197, 2018.

\bibitem[Dean et~al.(2019)Dean, Tu, Matni, and Recht]{dean2019safely}
S.~Dean, S.~Tu, N.~Matni, and B.~Recht.
\newblock Safely learning to control the constrained linear quadratic
  regulator.
\newblock In \emph{2019 American Control Conference (ACC)}, pages 5582--5588.
  IEEE, 2019.

\bibitem[Khammash(1990)]{khammash1990stability}
M.~H. Khammash.
\newblock \emph{Stability and performance robustness of discrete-time systems
  with structured uncertainty}.
\newblock PhD thesis, Rice University, 1990.

\bibitem[{Langbort} et~al.(2004){Langbort}, {Chandra}, and
  {D'Andrea}]{langbort2004distributed}
C.~{Langbort}, R.~S. {Chandra}, and R.~{D'Andrea}.
\newblock Distributed control design for systems interconnected over an
  arbitrary graph.
\newblock \emph{IEEE Transactions on Automatic Control}, 49\penalty0
  (9):\penalty0 1502--1519, Sep. 2004.
\newblock \doi{10.1109/TAC.2004.834123}.

\bibitem[Lessard(2014)]{lessard2014state}
L.~Lessard.
\newblock State-space solution to a minimum-entropy $h_\infty$-optimal control
  problem with a nested information constraint.
\newblock In \emph{53rd IEEE Conference on Decision and Control}, pages
  4026--4031. IEEE, 2014.

\bibitem[Mahajan et~al.(2012)Mahajan, Martins, Rotkowitz, and
  Yuksel]{2012_Mahajan_Info_survey}
A.~Mahajan, N.~Martins, M.~Rotkowitz, and S.~Yuksel.
\newblock Information structures in optimal decentralized control.
\newblock In \emph{Decision and Control (CDC), 2012 IEEE 51st Annual Conference
  on}, pages 1291--1306, 2012.
\newblock \doi{10.1109/CDC.2012.6425819}.

\bibitem[Matni(2014)]{matni2014distributed}
N.~Matni.
\newblock Distributed control subject to delays satisfying an $h_\infty$ norm
  bound.
\newblock In \emph{53rd IEEE Conference on Decision and Control}, pages
  4006--4013. IEEE, 2014.

\bibitem[Matni et~al.(2017)Matni, Wang, and Anderson]{matni2017scalable}
N.~Matni, Y.-S. Wang, and J.~Anderson.
\newblock Scalable system level synthesis for virtually localizable systems.
\newblock In \emph{2017 IEEE 56th Annual Conference on Decision and Control
  (CDC)}, pages 3473--3480. IEEE, 2017.

\bibitem[Megretski and Rantzer(1997)]{megretski1997system}
A.~Megretski and A.~Rantzer.
\newblock System analysis via integral quadratic constraints.
\newblock \emph{IEEE Transactions on Automatic Control}, 42\penalty0
  (6):\penalty0 819--830, 1997.

\bibitem[Meissen et~al.(2015)Meissen, Lessard, Arcak, and
  Packard]{meissen2015compositional}
C.~Meissen, L.~Lessard, M.~Arcak, and A.~K. Packard.
\newblock Compositional performance certification of interconnected systems
  using admm.
\newblock \emph{Automatica}, 61:\penalty0 55--63, 2015.

\bibitem[Packard and Doyle(1993)]{packard1993complex}
A.~Packard and J.~Doyle.
\newblock The complex structured singular value.
\newblock \emph{Automatica}, 29\penalty0 (1):\penalty0 71--109, 1993.

\bibitem[Rotkowitz and Lall(2006)]{2006_Rotkowitz_QI_TAC}
M.~Rotkowitz and S.~Lall.
\newblock A characterization of convex problems in decentralized control.
\newblock \emph{Automatic Control, IEEE Transactions on}, 51\penalty0
  (2):\penalty0 274--286, 2006.

\bibitem[{Rösinger} and {Scherer}(2017)]{rosinger2017structured}
C.~A. {Rösinger} and C.~W. {Scherer}.
\newblock Structured controller design with applications to networked systems.
\newblock In \emph{2017 IEEE 56th Annual Conference on Decision and Control
  (CDC)}, pages 4771--4776, Dec 2017.
\newblock \doi{10.1109/CDC.2017.8264365}.

\bibitem[Wang and Matni(2016)]{wang2016localized}
Y.-S. Wang and N.~Matni.
\newblock Localized lqg optimal control for large-scale systems.
\newblock In \emph{2016 American Control Conference (ACC)}, pages 1954--1961.
  IEEE, 2016.

\bibitem[Wang et~al.(2014)Wang, Matni, and Doyle]{wang2014localized}
Y.-S. Wang, N.~Matni, and J.~C. Doyle.
\newblock Localized lqr optimal control.
\newblock In \emph{53rd IEEE Conference on Decision and Control}, pages
  1661--1668. IEEE, 2014.

\bibitem[Wang et~al.(2018)Wang, Matni, and Doyle]{wang2018separable}
Y.-S. Wang, N.~Matni, and J.~C. Doyle.
\newblock Separable and localized system-level synthesis for large-scale
  systems.
\newblock \emph{IEEE Transactions on Automatic Control}, 63\penalty0
  (12):\penalty0 4234--4249, 2018.

\bibitem[Wang et~al.(2019)Wang, Matni, and Doyle]{wang2019system}
Y.-S. Wang, N.~Matni, and J.~C. Doyle.
\newblock A system level approach to controller synthesis.
\newblock \emph{IEEE Transactions on Automatic Control}, 2019.

\bibitem[Zheng et~al.(2019)Zheng, Furieri, Papachristodoulou, Li, and
  Kamgarpour]{zheng2019equivalence}
Y.~Zheng, L.~Furieri, A.~Papachristodoulou, N.~Li, and M.~Kamgarpour.
\newblock On the equivalence of youla, system-level and input-output
  parameterizations.
\newblock \emph{arXiv preprint arXiv:1907.06256}, 2019.

\bibitem[Zhou et~al.(1996)Zhou, Doyle, and Glover]{zhou1996robust}
K.~Zhou, J.~C. Doyle, and K.~Glover.
\newblock \emph{Robust and optimal control}.
\newblock Prentice Hall New Jersey, 1996.

\end{thebibliography}
\end{small} 
\end{document}